\documentclass[12pt]{article}
\usepackage{amsmath,amssymb,amsthm}
\numberwithin{equation}{section}
\usepackage{hyperref}
\usepackage{units}
\usepackage{color}
\usepackage[T1]{fontenc}
\usepackage[utf8]{inputenc}
\usepackage{authblk}
\usepackage{bm}
\usepackage{graphicx}
\usepackage{enumitem}
\usepackage{datetime}
\usepackage{color}

\usepackage[toc,page]{appendix}

\newcommand{\Z}{{\mathbb Z}}

\newtheorem{thm}{Theorem}[section]

\newtheorem{lemma}[thm]{Lemma}

\title{Factored closed-form expressions for the sums of cubes of Fibonacci and Lucas numbers\thanks{AMS Classification Numbers : 11B37, 11B39}\vspace{10mm}}

\author[]{Kunle Adegoke \thanks{adegoke00@gmail.com, kunle.adegoke@yandex.com}}

\affil{Department of Physics and Engineering Physics, \mbox{Obafemi Awolowo University}, Ile-Ife, Nigeria}

\begin{document}

\date{}

\maketitle

\begin{abstract}
\noindent We obtain {\em explicit} factored closed-form  expressions for Fibonacci and Lucas sums of the form \mbox{$\sum_{k = 1}^n {F_{2rk}^3 }$} and \mbox{$\sum_{k = 1}^n {L_{2rk}^3 }$}, where $r$~and~$n$ are integers. 
\end{abstract}

\section{Introduction}

The Fibonacci numbers, $F_n$, and Lucas numbers, $L_n$, are defined, for \mbox{$n\in\Z$}, as usual, through the recurrence relations \mbox{$F_n=F_{n-1}+F_{n-2}$}, \mbox{$F_0=0$, $F_1=1$} and \mbox{$L_n=L_{n-1}+L_{n-2}$}, $L_0=2$, $L_1=1$, with $F_{-n}=(-1)^{n-1}F_n$ and $L_{-n}=(-1)^nL_n$.

\bigskip

Clary and Hemenway~\cite{clary} derived the remarkable formulas
\begin{equation}\label{equ.mnyq31m}
4\sum_{k=1}^nF_{2k}^3=
\begin{cases}
F_n^2L_{n+1}^2F_{n-1}L_{n+2} & \text{if $n$ is even,}\\
L_n^2F_{n+1}^2L_{n-1}F_{n+2} & \text{if $n$ is odd\,,}
\end{cases}
\end{equation}
and
\begin{equation}\label{equ.gd1awb5}
8\sum_{k = 1}^n {F_{4k}^3 }  = F_{2n}^2 F_{2n + 2}^2 (L_{4n + 2}  + 6)\,.
\end{equation}
In this present paper we will derive the following corresponding Lucas counterparts of~\eqref{equ.mnyq31m} and~\eqref{equ.gd1awb5}:
\begin{equation}\label{equ.aa5tal0}
4\sum_{k=1}^nL_{2k}^3=
\begin{cases}
5F_nF_{n+1}(L_nL_{n+1}L_{2n+1}+16) & \text{if $n$ is even,}\\
L_nL_{n+1}(5F_nF_{n+1}L_{2n+1}+16) & \text{if $n$ is odd\,,}
\end{cases}
\end{equation}
and
\begin{equation}
8\sum_{k = 1}^n {L_{4k}^3 }  = F_{2n} L_{2n + 2} (5L_{2n}F_{2n + 2}F_{4n + 2}  + 32)\,.
\end{equation}
In fact we will derive the following more general results:
\begin{itemize}
\item If $r$ is odd, then
\begin{equation}\label{equ.ewu6n75}
L_{3r}\sum_{k=1}^nF_{2rk}^3=
\begin{cases}
F_{rn}^2L_{rn+r}^2(L_{rn}F_{rn+r}+F_r) & \text{if $n$ is even,}\\
L_{rn}^2F_{rn+r}^2(F_{rn}L_{rn+r}+F_r) & \text{if $n$ is odd\,,}
\end{cases}
\end{equation}
\begin{equation}\label{equ.bnbadtf}
L_{3r}\sum_{k=1}^nL_{2rk}^3=
\begin{cases}
5F_{rn}F_{rn+r}(L_{rn}L_{rn+r}L_{2rn+r}+4(L_{2r}+1)) & \text{if $n$ is even,}\\
L_{rn}L_{rn+r}(5F_{rn}F_{rn+r}L_{2rn+r}+4(L_{2r}+1)) & \text{if $n$ is odd\,.}
\end{cases}
\end{equation}
\item If $r$ is even, then
\begin{equation}\label{equ.px681ib}
F_{3r} \sum_{k = 1}^n {F_{2rk}^3 }  = F_{rn}^2 F_{rn + r}^2 (L_{rn} L_{rn + r}  + L_r )\,,
\end{equation}
\begin{equation}
F_{3r} \sum_{k = 1}^n {L_{2rk}^3 }  = F_{rn} L_{rn + r} (5L_{rn}F_{rn + r}F_{2rn + r}   + 4(L_{2r}  + 1))\,.
\end{equation}

\end{itemize}

As variations on identities~\eqref{equ.ewu6n75} and~\eqref{equ.px681ib} we will prove
\begin{itemize}
\item If $r$ is odd, then
\[
L_{3r}\sum_{k=1}^nF_{2rk}^3=
\begin{cases}
F_{rn}L_{rn+r}(L_{rn}F_{rn+r}F_{2rn+r}-2F_r^2) & \text{if $n$ is even,}\\
L_{rn}F_{rn+r}(F_{rn}L_{rn+r}F_{2rn+r}-2F_r^2) & \text{if $n$ is odd\,.}
\end{cases}
\]

\item If $r$ is even, then
\[
5F_{3r} \sum_{k = 1}^n {F_{2rk}^3 }  =  F_{rn} F_{rn + r}(L_{rn} L_{rn + r} L_{2rn + r}  - 2L_r^2)\,.
\]
\end{itemize}

\section{Required identities and preliminary results}

\subsection{Telescoping summation identity}
The following telescoping summation identity is a special case of more general identities proved in~\cite{adegoke}.
\begin{lemma}\label{finall}
If $f(k)$ is a real sequence and $m$, $q$ and $n$ are positive integers, then
\[
\sum_{k = 1}^n {\left[ {f(mk + mq) - f(mk)} \right]}  = \sum_{k = 1}^q {f(mk + mn)}  - \sum_{k = 1}^q {f(mk)}\,. 
\]

\end{lemma}

\subsection{First-power Fibonacci summation identities}
\begin{lemma}\label{thm.esyph89}
If $r$ and $n$ are integers, then
\begin{enumerate}
\item[(i)] If $r$ is even, then
\[
F_r\sum_{k=1}^nF_{2rk}=F_{rn}F_{rn+r}\,,
\] 
\item[(ii)] if $r$ is odd, then
\[
L_r\sum_{k=1}^nF_{2rk}=
\begin{cases}
F_{rn}L_{rn+r} & \text{if $n$ is even,}\\
L_{rn}F_{rn+r} & \text{if $n$ is odd.}
\end{cases}
\]
\end{enumerate}
\end{lemma}
\begin{proof}
Setting $v=2r$ and $u=2rk$ in the identity
\begin{equation}
L_{u + v} - (-1)^vL_{u-v}=5F_uF_v\label{equ.q79u3q4}
\end{equation}
gives
\begin{equation}\label{equ.jv8ztfh}
L_{2rk + 2r}  - L_{2rk - 2r}  = 5F_{2r} F_{2rk}\,. 
\end{equation}
Taking \mbox{$f(k)=L_{k-2r}$}, $q=2$ and $m=2r$ in Lemma~\ref{finall} and employing identity~\eqref{equ.jv8ztfh} we have
\begin{equation}\label{equ.w8cpa7u}
\begin{split}
5F_{2r} \sum_{k = 1}^n {F_{2rk} }  &= \sum_{k = 1}^2 {L_{2rk + 2rn - 2r} }  - \sum_{k = 1}^2 {L_{2rk - 2r} }\\
&= L_{2rn + 2r}  + L_{2rn}  - L_{2r}  - 2\,.
\end{split}
\end{equation}
If $r$ is even, then on account of the identity
\begin{equation}
L_{u + v} + (-1)^vL_{u-v}=L_uL_v \label{equ.ztsb3uk},\\ 
\end{equation}
we have
\[
L_{2rn+2r}+L_{2rn}=L_rL_{2rn+r},\quad L_{2r}+2=L_r^2\,,
\]
and since
\begin{equation}\label{equ.gv7zoe2}
F_{2u}=F_uL_u\,,
\end{equation}
identity~\eqref{equ.w8cpa7u} now becomes
\begin{equation}
\begin{split}
5F_{r} \sum_{k = 1}^n {F_{2rk} }  &= L_{2rn + r}  - L_r\\
&=5F_{rn}F_{rn+r}\,,\quad\mbox{by~\eqref{equ.q79u3q4}}\,,
\end{split}
\end{equation}
that is,
\[
F_{r} \sum_{k = 1}^n {F_{2rk} }=F_{rn}F_{rn+r},\quad\mbox{$r$ even}\,,
\]
and the first part of Lemma~\ref{thm.esyph89} is proved.

\bigskip

If $r$ is odd, then on account of the identities~\eqref{equ.q79u3q4} and~\eqref{equ.ztsb3uk}, we have
\[
L_{2rn+2r}+L_{2rn}=5F_rF_{2rn+r},\quad L_{2r}+2=5F_r^2\,,
\]
and identity~\eqref{equ.w8cpa7u} reduces to
\begin{equation}
\begin{split}
L_r\sum_{k=1}^nF_{2rk}&=F_{2rn+r}-F_r\\
&\quad=
\begin{cases}
F_{rn}L_{rn+r} & \text{if $n$ is even,}\\
L_{rn}F_{rn+r} & \text{if $n$ is odd,}
\end{cases}
\end{split}
\end{equation}
and the second part of Lemma~\ref{thm.esyph89} is proved.
In the last stage of the above derivation we made use of the identities
\begin{equation}\label{equ.yb05ue2}
F_{u + v} - (-1)^vF_{u-v}=F_vL_u 
\end{equation}
and
\begin{equation}\label{equ.jjkzwa0}
F_{u + v} + (-1)^vF_{u-v}=L_vF_u\,. 
\end{equation}
\end{proof}
\subsection{First-power Lucas summation identities}
\begin{lemma}\label{thm.hnf079l}
If $r$ and $n$ are integers, then
\begin{enumerate}
\item[(i)] If $r$ is even, then
\[
F_r\sum_{k=1}^nL_{2rk}=F_{rn}L_{rn+r}\,,
\] 
\item[(ii)] if $r$ is odd, then
\[
L_r\sum_{k=1}^nL_{2rk}=
\begin{cases}
5F_{rn}F_{rn+r} & \text{if $n$ is even,}\\
L_{rn}L_{rn+r} & \text{if $n$ is odd.}
\end{cases}
\]
\end{enumerate}
\end{lemma}
\begin{proof}
Setting $v=2r$ and $u=2rk$ in the identity~\eqref{equ.yb05ue2} gives
\begin{equation}\label{equ.gimtkiy}
F_{2rk + 2r}  - F_{2rk - 2r}  = F_{2r} L_{2rk}\,. 
\end{equation}
Taking \mbox{$f(k)=F_{k-2r}$}, $q=2$ and $m=2r$ in Lemma~\ref{finall} and employing identity~\eqref{equ.gimtkiy} we have
\begin{equation}\label{equ.jhbv3g7}
\begin{split}
F_{2r} \sum_{k = 1}^n {L_{2rk} }  &= \sum_{k = 1}^2 {F_{2rk + 2rn - 2r} }  - \sum_{k = 1}^2 {F_{2rk - 2r} }\\ 
&= F_{2rn+2r}+F_{2rn}-F_{2r}\,.
\end{split}
\end{equation}
If $r$ is even, then choosing $v=r$ and $u=2rn+r$ in identity~\eqref{equ.jjkzwa0} gives
\begin{equation}
F_{2rn+2r}+F_{2rn}=L_rF_{2rn+r}
\end{equation}
and, on account of identity~\eqref{equ.gv7zoe2}, the identity~\eqref{equ.jhbv3g7} reduces to
\begin{equation}
\begin{split}
F_r\sum_{k=1}^nL_{2rk}&=F_{2rn+r}-F_r\\
&=F_{rn+r+rn}-F_{rn+r-rn}\\
&=F_{rn}L_{rn+r}\,,\quad\mbox{by identity~\eqref{equ.yb05ue2},}
\end{split}
\end{equation}

and the first part of Lemma~\ref{thm.hnf079l} is proved.

\bigskip

If $r$ is odd, then choosing $v=r$ and $u=2rn+r$ in identity~\eqref{equ.yb05ue2} gives
\begin{equation}
F_{2rn+2r}+F_{2rn}=F_rL_{2rn+r}
\end{equation}
and, again on account of identity~\eqref{equ.gv7zoe2}, the identity~\eqref{equ.jhbv3g7} now reduces to
\begin{equation}
\begin{split}
L_r\sum_{k=1}^nL_{2rk}&=L_{2rn+r}-L_r\\
&=L_{rn+r+rn}-L_{rn+r-rn}\\
&=
\begin{cases}
5F_{rn}F_{rn+r} & \text{if $n$ is even,}\\
L_{rn}L_{rn+r} & \text{if $n$ is odd,}
\end{cases}
\end{split}
\end{equation}
where in the last step we used the identities~\eqref{equ.q79u3q4} and~\eqref{equ.ztsb3uk}.

\end{proof}
\subsection{Other identities}
\begin{lemma}\label{thm.w3ql3gi}
If $r$ and $n$ are integers, then
\[
\frac{F_{3rn}F_{3rn+3r}}{F_{rn}F_{rn+r}}=L_{rn}L_{rn+r}L_{2rn+r}+L_{2r}+(-1)^{r-1}.
\]
\end{lemma}
\begin{proof}
Using the identity (equation~(36) of~\cite{clary}, also (3.3) of~\cite{dresel})
\begin{equation}\label{equ.oyvdke0}
F_{3u}=5F_u^3+3(-1)^uF_u\,,
\end{equation}
we have
\begin{equation}\label{equ.gs4xocq}
\begin{split}
\frac{{F_{3rn} F_{3rn + 3r} }}{{F_{rn} F_{rn + r} }} &= (5F_{rn}^2  + 3( - 1)^{rn} )(5F_{rn + r}^2  + 3( - 1)^{rn + r} )\\
&= (L_{rn}^2  - ( - 1)^{rn} )(L_{rn + r}^2  - ( - 1)^{rn + r} )\\
&= L_{rn}^2 L_{rn + r}^2  - ( - 1)^{rn + r} L_{rn}^2  - ( - 1)^{rn} L_{rn + r}^2  + ( - 1)^r\,, 
\end{split}
\end{equation}
where we have also made use of the identity
\begin{equation}\label{equ.mpb5qh0}
5F_u^2-L_u^2=(-1)^{u-1}4\,.
\end{equation}
Now,
\[
\begin{split}
L_{rn}^2 L_{rn + r}^2&= L_{rn} L_{rn + r} (L_{rn} L_{rn + r} )\\ 
&= L_{rn} L_{rn + r} (L_{2rn + r}  + ( - 1)^{rn} L_r )\quad\mbox{by~\eqref{equ.ztsb3uk}}\\
&= L_{rn} L_{rn + r} L_{2rn + r}  + ( - 1)^{rn} L_{rn} L_{rn + r} L_r \,.
\end{split}
\]
Therefore
\[
\begin{split}
\frac{{F_{3rn} F_{3rn + 3r} }}{{F_{rn} F_{rn + r} }} &= L_{rn} L_{rn + r} L_{2rn + r}\\ 
&\quad + ( - 1)^{rn} L_{rn + r} (L_{rn}L_r  - L_{rn + r} )\\
&\qquad - ( - 1)^{rn + r} L_{rn}^2  + ( - 1)^r\,. 
\end{split}
\]
But
\[
\begin{split}
&( - 1)^{rn} L_{rn + r} (L_{rn} L_r  - L_{rn + r} )\\
 &= ( - 1)^{rn} L_{rn + r} (L_{rn + r}  + ( - 1)^r L_{rn - r}  - L_{rn + r} ),\quad\mbox{by~\eqref{equ.ztsb3uk}}\\
& = ( - 1)^{rn + r} L_{rn + r} L_{rn - r}\\ 
&= ( - 1)^{rn + r} (L_{2rn}  + ( - 1)^{rn - r} L_{2r} ),\quad\mbox{again by~\eqref{equ.ztsb3uk}}\\
&= ( - 1)^{rn + r} L_{2rn}  + L_{2r}\,. 
\end{split}
\]
Thus
\[
\begin{split}
\frac{{F_{3rn} F_{3rn + 3r} }}{{F_{rn} F_{rn + r} }} &= L_{rn} L_{rn + r} L_{2rn + r}\\ 
&\quad + ( - 1)^{rn + r} L_{2rn}  + L_{2r}\\ 
&\qquad - ( - 1)^{rn + r} L_{rn}^2  + ( - 1)^r 
\end{split}
\]
\[
\begin{split}
&\qquad= L_{rn} L_{rn + r} L_{2rn + r}\\ 
&\qquad\quad + ( - 1)^{rn + r} (L_{2rn}  - L_{rn}^2 )\\
&\qquad\qquad + L_{2r}  + ( - 1)^r\,. 
\end{split}
\]
Finally, using the identity
\begin{equation}\label{equ.fpjkos2}
L_{2u}=L_u^2+(-1)^{u-1}2\,,
\end{equation}
obtained by setting $v=u$ in identity~\eqref{equ.ztsb3uk}, we have the statement of the Lemma.
\end{proof}
\begin{lemma}\label{thm.cdcw5y2}
If $r$ and $n$ are integers, then
\[
\frac{L_{3rn}L_{3rn+3r}}{L_{rn}L_{rn+r}}=5F_{rn}F_{rn+r}L_{2rn+r}+L_{2r}+(-1)^{r-1}.
\]
\end{lemma}
\begin{proof}
Using the following identity, (equation (1.6) of~\cite{dresel})
\begin{equation}\label{equ.aiunttr}
L_{3u}=L_u^3-3(-1)^uL_u\,,
\end{equation}
we have
\[
\begin{split}
\frac{{L_{3rn} L_{3rn + 3r} }}{{L_{rn} L_{rn + r} }} &= (L_{rn}^2  - 3( - 1)^{rn} )(L_{rn + r}^2  - 3( - 1)^{rn + r} )\\
&= (5F_{rn}^2  + ( - 1)^{rn} )(5F_{rn + r}^2  + ( - 1)^{rn + r} ),\quad\mbox{by~\eqref{equ.mpb5qh0}}\\
&= 25F_{rn}^2 F_{rn + r}^2  + ( - 1)^{rn + r} 5F_{rn}^2  + ( - 1)^{rn} 5F_{rn + r}^2  + ( - 1)^r\,, 
\end{split}
\]
and the rest of the calculation then proceeds as in the proof of Lemma~\ref{thm.w3ql3gi}, the basic required identities now being~\eqref{equ.q79u3q4}, \eqref{equ.jjkzwa0} and the identity
\begin{equation}\label{equ.gr4yegj}
L_{2u}=5F_u^2+(-1)^u2,
\end{equation}
obtained by setting $v=u$ in identity~\eqref{equ.q79u3q4}.
\end{proof}
\begin{lemma}\label{thm.ip29mgz}
If $r$ and $n$ are integers, then
\[
\frac{L_{3rn}F_{3rn+3r}}{L_{rn}F_{rn+r}}=5F_{rn}L_{rn+r}F_{2rn+r}+L_{2r}+(-1)^r.
\]
\end{lemma}
\begin{lemma}\label{thm.c1bi001}
If $r$ and $n$ are integers, then
\[
\frac{F_{3rn}L_{3rn+3r}}{F_{rn}L_{rn+r}}=5L_{rn}F_{rn+r}F_{2rn+r}+L_{2r}+(-1)^r.
\]
\end{lemma}
Different but equivalent versions of Lemmata~\ref{thm.w3ql3gi}---\ref{thm.c1bi001} are given below:
\begin{lemma}\label{thm.c6tpo1r}
If $r$ and $n$ are integers, then
\[
\frac{{F_{3rn} F_{3rn + 3r} }}{{F_{rn} F_{rn + r} }} = L_{2rn + r}^2  + ( - 1)^{nr} L_{rn + r}^2  + ( - 1)^{(n - 1)r} L_{rn}^2  + L_r^2  + ( - 1)^{r - 1} 7\,.
\]
\end{lemma}
\begin{proof}
The proof is similar to that of Lemma~\ref{thm.w3ql3gi}, but here we use
\[
\begin{split}
L_{rn}^2 L_{rn + r}^2  &= (L_{2rn + r}  + ( - 1)^{rn} L_r )^2\\ 
 &= L_{2rn + r}^2  + L_r^2  + 2( - 1)^{rn} \{L_r L_{2rn + r}\}\\ 
 &= L_{2rn + r}^2  + L_r^2  + 2( - 1)^{rn} \{L_{2rn + 2r}  + ( - 1)^r L_{2rn} \}\\
 &= L_{2rn + r}^2  + L_r^2  + 2( - 1)^{rn} \{ L_{rn + r}^2  + ( - 1)^{rn + r - 1} 2\\
 &\qquad+ ( - 1)^r [L_{rn}^2  + ( - 1)^{rn - 1} 2]\} 
\end{split}
\]
and substitute in~\eqref{equ.gs4xocq}.
\end{proof}
\begin{lemma}\label{thm.da8g4ih}
If $r$ and $n$ are integers, then
\[
\frac{{L_{3rn} L_{3rn + 3r} }}{{L_{rn} L_{rn + r} }} = L_{2rn + r}^2  + ( - 1)^{nr-1} L_{rn + r}^2  - ( - 1)^{(n - 1)r} L_{rn}^2  + L_r^2  + ( - 1)^r\,.
\]
\end{lemma}
\begin{lemma}\label{thm.kquo1a5}
If $r$ and $n$ are integers, then
\[
\frac{{L_{3rn} F_{3rn + 3r} }}{{L_{rn} F_{rn + r} }} = 5F_{2rn + r}^2  + ( - 1)^{nr - 1} 5F_{rn + r}^2  + ( - 1)^{(n - 1)r} 5F_{rn}^2  + 5F_r^2  + ( - 1)^r 3\,.
\]

\end{lemma}
\begin{lemma}\label{thm.powhazn}
If $r$ and $n$ are integers, then
\[
\frac{{F_{3rn} L_{3rn + 3r} }}{{F_{rn} L_{rn + r} }} = 5F_{2rn + r}^2  + ( - 1)^{nr} 5F_{rn + r}^2  - ( - 1)^{(n - 1)r} 5F_{rn}^2  + 5F_r^2  + ( - 1)^r 3\,.
\]
\end{lemma}

\section{Main results}
\subsection{Sums of cubes of Fibonacci numbers}
\begin{thm}\label{thm.antrd3y}
If $r$ and $n$ are integers such that $r$ is odd, then
\[
L_{3r}\sum_{k=1}^nF_{2rk}^3=
\begin{cases}
F_{rn}L_{rn+r}(L_{rn}F_{rn+r}F_{2rn+r}-2F_r^2) & \text{if $n$ is even,}\\
L_{rn}F_{rn+r}(F_{rn}L_{rn+r}F_{2rn+r}-2F_r^2) & \text{if $n$ is odd\,.}
\end{cases}
\]
\end{thm}
\begin{proof}
Setting $u=2rk$ in identity~\eqref{equ.oyvdke0} and summing, we have
\[
5\sum_{k = 1}^n {F_{2rk}^3 }  = \sum_{k = 1}^n {F_{6rk} }  - 3\sum_{k = 1}^n {F_{2rk} }\,,
\]
so that,
\begin{equation}\label{equ.dpiq8yz}
\begin{split}
5L_{3r} \sum_{k = 1}^n {F_{2rk}^3 } &= L_{3r} \sum_{k = 1}^n {F_{6rk} }  - 3\frac{{L_{3r} }}{{L_r }}L_r \sum_{k = 1}^n {F_{2rk} }\\ 
&= L_{3r} \sum_{k = 1}^n {F_{6rk} }  - 3(L_r^2  + 3)L_r \sum_{k = 1}^n {F_{2rk} }\,. 
\end{split}
\end{equation}
\begin{itemize}
\item If $n$ is even, then, by Lemma~\ref{thm.esyph89}, identity~\eqref{equ.dpiq8yz} can be written as
\[
5L_{3r} \sum_{k = 1}^n {F_{2rk}^3 }  = F_{3rn} L_{3rn + 3r}  - 3(L_r^2  + 3)F_{rn} L_{rn + r}\,, 
\]
so that
\[
\begin{split}
\frac{{5L_{3r} \sum_{k = 1}^n {F_{2rk}^3 } }}{{F_{rn} L_{rn + r} }} &= \frac{{F_{3rn} L_{3rn + 3r} }}{{F_{rn} L_{rn + r} }} - 3(L_r^2  + 3)\\
&= 5L_{rn} F_{rn + r} F_{2rn + r}  + L_{2r}  - 1 - 3L_r^2  - 9,\quad\mbox{by Lemma~\ref{thm.c1bi001}}\\
&= 5L_{rn} F_{rn + r} F_{2rn + r}  - 10F_r^2,\quad\mbox{by \eqref{equ.mpb5qh0} and \eqref{equ.fpjkos2}}\,. 
\end{split}
\]
\item If $n$ is odd, then, by Lemma~\ref{thm.esyph89}, we have
\[
5L_{3r} \sum_{k = 1}^n {F_{2rk}^3 }  = L_{3rn} F_{3rn + 3r}  - 3(L_r^2  + 3)L_{rn} F_{rn + r}\,, 
\]
so that
\[
\begin{split}
\frac{{5L_{3r} \sum_{k = 1}^n {F_{2rk}^3 } }}{{L_{rn} F_{rn + r} }} &= \frac{{L_{3rn} F_{3rn + 3r} }}{{L_{rn} F_{rn + r} }} - 3(L_r^2  + 3)\\
&= 5F_{rn} L_{rn + r} F_{2rn + r}  + L_{2r}  - 1 - 3L_r^2  - 9,\quad\mbox{by Lemma~\ref{thm.c1bi001}}\\
&= 5F_{rn} L_{rn + r} F_{2rn + r}  - 10F_r^2,\quad\mbox{by \eqref{equ.mpb5qh0} and \eqref{equ.fpjkos2}}\,. 
\end{split}
\]
\end{itemize}
\end{proof}
\begin{thm}
If $r$ and $n$ are integers such that $r$ is even, then
\[
5F_{3r} \sum_{k = 1}^n {F_{2rk}^3 }  =  F_{rn} F_{rn + r}(L_{rn} L_{rn + r} L_{2rn + r}  - 2L_r^2)\,.
\]
\end{thm}
\begin{proof}
\[
\begin{split}
5F_{3r} \sum_{k = 1}^n {F_{2rk}^3 }  &= F_{3r} \sum_{k = 1}^n {F_{6rk} }  - 3\frac{{F_{3r} }}{{F_r }}F_r \sum_{k = 1}^n {F_{2rk} }\\
&= F_{3rn} F_{3rn + 3r}  - 3(5F_r^2  + 3)F_{rn} F_{rn + r},\\
& \mbox{by Lemma \ref{thm.esyph89} and identity \eqref{equ.oyvdke0}}\,,
\end{split}
\]
so that
\[
\begin{split}
\frac{{5F_{3r} \sum_{k = 1}^n {F_{2rk}^3 } }}{{F_{rn} F_{rn + r} }} &= \frac{{F_{3rn} F_{3rn + 3r} }}{{F_{rn} F_{rn + r} }} - 3(5F_r^2  + 3)\\
&= L_{rn} L_{rn + r} L_{2rn + r}  + L_{2r}  - 1 - 15F_r^2  - 9\\
&\mbox{(by Lemma \ref{thm.w3ql3gi} and identity \eqref{equ.oyvdke0})}\,,\\
&= L_{rn} L_{rn + r} L_{2rn + r}  - 2L_r^2,\quad\mbox{by \eqref{equ.mpb5qh0}, \eqref{equ.fpjkos2} and \eqref{equ.gr4yegj}}\,.
\end{split}
\]

\end{proof}
\begin{thm}
If $r$ and $n$ are integers such that $r$ is odd, then
\[
L_{3r}\sum_{k=1}^nF_{2rk}^3=
\begin{cases}
F_{rn}^2L_{rn+r}^2(L_{rn}F_{rn+r}+F_r) & \text{if $n$ is even,}\\
L_{rn}^2F_{rn+r}^2(F_{rn}L_{rn+r}+F_r) & \text{if $n$ is odd\,.}
\end{cases}
\]
\end{thm}
\begin{proof}
\begin{itemize}
\item If $n$ is even, then from Lemma~\ref{thm.esyph89} and identity~\eqref{equ.dpiq8yz} we have
\[
\begin{split}
\frac{{5L_{3r} \sum_{k = 1}^n {F_{2rk}^3 } }}{{F_{rn} L_{rn + r} }} &= \frac{{F_{3rn} L_{3rn + 3r} }}{{F_{rn} L_{rn + r} }} - 3(L_r^2  + 3)\\
&= 5F_{2rn + r}^2  + 5F_{rn + r}^2  + 5F_{rn}^2  + 5F_r^2  - 3 - 3L_r^2  - 9,\,\mbox{by Lemma~\ref{thm.powhazn}}\\
&= 5F_{2rn + r}^2  + 5F_{rn + r}^2  + 5F_{rn}^2  - 10F_r^2\quad\mbox{by identity~\eqref{equ.mpb5qh0}}\,,
\end{split}
\]
so that
\[
\begin{split}
\frac{{L_{3r} \sum_{k = 1}^n {F_{2rk}^3 } }}{{F_{rn} L_{rn + r} }} &= F_{2rn + r}^2  + F_{rn + r}^2  + F_{rn}^2  - 2F_r^2\\
&= (F_{2rn + r}^2  - F_r^2 ) + (F_{rn + r}^2+ F_{rn}^2)  - F_r^2\,.
\end{split}
\]
Using the following identity, derived in \cite{howard},
\begin{equation}
F_u^2 + (-1)^{u+v-1}F_v^2=F_{u-v}F_{u+v}\,,\label{equ.wwhhpdy}  
\end{equation}
we have
\[
\begin{split}
\frac{{L_{3r} \sum_{k = 1}^n {F_{2rk}^3 } }}{{F_{rn} L_{rn + r} }}&= F_{2rn} F_{2rn + 2r}  + F_r F_{2rn + r}  - F_r^2\\
&= F_{2rn} F_{2rn + 2r}  + F_r (F_{2rn + r}  - F_r )\\
&= F_{2rn} F_{2rn + 2r}  + F_r F_{rn} L_{rn + r},\,\mbox{by identity~\eqref{equ.yb05ue2}}\\ 
&= F_{rn} L_{rn + r} L_{rn} F_{rn + r}  + F_r F_{rn} L_{rn + r}\\ 
&= F_{rn} L_{rn + r} (L_{rn} F_{rn + r}  + F_r )\,.
\end{split}
\]
\item If $n$ is odd, then from Lemma~\ref{thm.esyph89} and identity~\eqref{equ.dpiq8yz} we have
\[
\begin{split}
\frac{{5L_{3r} \sum_{k = 1}^n {F_{2rk}^3 } }}{{L_{rn} F_{rn + r} }} &= \frac{{L_{3rn} F_{3rn + 3r} }}{{L_{rn} F_{rn + r} }} - 3(L_r^2  + 3)\\
&= 5F_{2rn + r}^2  + 5F_{rn + r}^2  + 5F_{rn}^2  + 5F_r^2  - 3 - 3L_r^2  - 9,\,\mbox{by Lemma~\ref{thm.kquo1a5}}\\
&= 5F_{2rn + r}^2  + 5F_{rn + r}^2  + 5F_{rn}^2  - 10F_r^2\quad\mbox{by identity~\eqref{equ.mpb5qh0}}\,,
\end{split}
\]
so that
\[
\begin{split}
\frac{{L_{3r} \sum_{k = 1}^n {F_{2rk}^3 } }}{{L_{rn} F_{rn + r} }} &= F_{2rn + r}^2  + F_{rn + r}^2  + F_{rn}^2  - 2F_r^2\\
&= (F_{2rn + r}^2  - F_r^2 ) + (F_{rn + r}^2+ F_{rn}^2)  - F_r^2\,.
\end{split}
\]
Using identity~\eqref{equ.wwhhpdy}, we have
\[
\begin{split}
\frac{{L_{3r} \sum_{k = 1}^n {F_{2rk}^3 } }}{{L_{rn} F_{rn + r} }}&= F_{2rn} F_{2rn + 2r}  + F_r F_{2rn + r}  - F_r^2\\
&= F_{2rn} F_{2rn + 2r}  + F_r (F_{2rn + r}  - F_r )\\
&= F_{2rn} F_{2rn + 2r}  + F_r L_{rn} F_{rn + r},\,\mbox{by identity~\eqref{equ.jjkzwa0}}\\ 
&= F_{rn} L_{rn + r} L_{rn} F_{rn + r}  + F_r L_{rn} F_{rn + r}\\ 
&= L_{rn} F_{rn + r} (F_{rn} L_{rn + r}  + F_r )\,.
\end{split}
\]

\end{itemize}
\end{proof}
\begin{thm}
If $r$ and $n$ are integers such that $r$ is even, then
\begin{equation}
F_{3r} \sum_{k = 1}^n {F_{2rk}^3 }  = F_{rn}^2 F_{rn + r}^2 (L_{rn} L_{rn + r}  + L_r )\,.
\end{equation}
\end{thm}
\begin{proof}
\[
\begin{split}
5F_{3r} \sum_{k = 1}^n {F_{2rk}^3 }  &= F_{3r} \sum_{k = 1}^n {F_{6rk} }  - 3\frac{{F_{3r} }}{{F_r }}F_r \sum_{k = 1}^n {F_{2rk} }\\
&= F_{3rn} F_{3rn + 3r}  - 3(5F_r^2  + 3)F_{rn} F_{rn + r}\,,
\end{split}
\]
so that
\[
\begin{split}
\frac{{5F_{3r} \sum_{k = 1}^n {F_{2rk}^3 } }}{{F_{rn} F_{rn + r} }} &= \frac{{F_{3rn} F_{3rn + 3r} }}{{F_{rn} F_{rn + r} }} - 3(5F_r^2  + 3)\\
&=L_{2rn+r}^2+L_{rn+r}^2+L_{rn}^2+L_r^2-7 - 15F_r^2  - 9,\quad\mbox{by Lemma~\ref{thm.c6tpo1r}}\\
&=L_{2rn+r}^2+L_{rn+r}^2-2L_r^2+5F_{rn}^2,\quad\mbox{by \eqref{equ.mpb5qh0}}\\
&=(L_{2rn+r}^2-L_r^2)+(L_{rn+r}^2-L_r^2)+5F_{rn}^2\,.
\end{split}
\]
Using the identity (derived in \cite{howard})
\begin{equation}
L_u^2 + (-1)^{u+v-1}L_v^2=5F_{u-v}F_{u+v}\,,\label{equ.wodrq78}  
\end{equation}
we see that
\begin{equation}
L_{2rn + r}^2  - L_r^2  = 5F_{2rn} F_{2rn + 2r}  = 5F_{rn} F_{rn + r} L_{rn} L_{rn + r}
\end{equation}
and
\begin{equation}
L_{rn + r}^2  - L_r^2  = 5F_{rn} F_{rn + 2r}\,.
\end{equation}
Thus,
\[
\begin{split}
\frac{{F_{3r} \sum_{k = 1}^n {F_{2rk}^3 } }}{{F_{rn} F_{rn + r} }} &= F_{rn} F_{rn + r} L_{rn} L_{rn + r}  + F_{rn} F_{rn + 2r}  + F_{rn}^2\\ 
&= F_{rn} F_{rn + r} L_{rn} L_{rn + r}  + F_{rn} (F_{rn}  + F_{rn + 2r} )\\
&= F_{rn} F_{rn + r} L_{rn} L_{rn + r}  + F_{rn} F_{rn + r} L_r,\quad\mbox{by identity \eqref{equ.jjkzwa0}}\\ 
&= F_{rn} F_{rn + r} (L_{rn} L_{rn + r}  + L_r )\,.
\end{split}
\]

\end{proof}

\subsection{Sums of cubes of Lucas numbers}
\begin{thm}\label{thm.antrd3y}
If $r$ and $n$ are integers such that $r$ is odd, then
\[
L_{3r}\sum_{k=1}^nL_{2rk}^3=
\begin{cases}
5F_{rn}F_{rn+r}(L_{rn}L_{rn+r}L_{2rn+r}+4(L_{2r}+1)) & \text{if $n$ is even,}\\
L_{rn}L_{rn+r}(5F_{rn}F_{rn+r}L_{2rn+r}+4(L_{2r}+1)) & \text{if $n$ is odd\,.}
\end{cases}
\]
\end{thm}
\begin{proof}
Using identity~\eqref{equ.aiunttr} with $u=2rk$, we have
\[
\sum_{k = 1}^n {L_{2rk}^3 }  = \sum_{k = 1}^n {L_{6rk} }  + 3\sum_{k = 1}^n {L_{2rk} }\,, 
\]
so that
\[
\begin{split}
L_{3r} \sum_{k = 1}^n {L_{2rk}^3 }  &= L_{3r} \sum_{k = 1}^n {L_{6rk} }  + 3\frac{{L_{3r} }}{{L_r }}L_r \sum_{k = 1}^n {L_{2rk} }\\ 
&= L_{3r} \sum_{k = 1}^n {L_{6rk} }  + 3(L_r^2  + 3)L_r \sum_{k = 1}^n {L_{2rk} },\quad\mbox{by \eqref{equ.aiunttr}}\,.\\ 
\end{split}
\]
\begin{itemize}
\item If $n$ is even, then by Lemma~\ref{thm.hnf079l} we have
\begin{equation}
L_{3r}\sum_{k = 1}^n {L_{2rk}^3 }  = 5F_{3rn} F_{3rn + 3r}  + 3(L_r^2  + 3)5F_{rn} F_{rn + r}\,,
\end{equation} 
so that
\[
\begin{split}
\frac{{L_{3r}\sum_{k = 1}^n {L_{2rk}^3 } }}{{5F_{rn} F_{rn + r} }} &= \frac{{F_{3rn} F_{3rn + 3r} }}{{F_{rn} F_{rn + r} }} + 3(L_r^2  + 3)\\
&= L_{rn} L_{rn + r} L_{2rn + r}  + L_{2r}  + 1 + 3L_r^2  + 9,\quad\mbox{by Lemma \ref{thm.w3ql3gi}}\\
&= L_{rn} L_{rn + r} L_{2rn + r}  + 4(L_{2r}  + 1),\quad\mbox{by \eqref{equ.fpjkos2}}\,.
\end{split}
\]
\item If $n$ is odd, then by Lemma~\ref{thm.hnf079l} we have
\begin{equation}
L_{3r}\sum_{k = 1}^n {L_{2rk}^3 }  = L_{3rn} L_{3rn + 3r}  + 3(L_r^2  + 3)L_{rn} L_{rn + r}\,,
\end{equation} 
so that
\[
\begin{split}
\frac{{L_{3r}\sum_{k = 1}^n {L_{2rk}^3 } }}{{L_{rn} L_{rn + r} }} &= \frac{{L_{3rn} L_{3rn + 3r} }}{{L_{rn} L_{rn + r} }} + 3(L_r^2  + 3)\\
&= 5F_{rn} F_{rn + r} L_{2rn + r}  + L_{2r}  + 1 + 3L_r^2  + 9,\quad\mbox{by Lemma \ref{thm.cdcw5y2}}\\
&= 5F_{rn} F_{rn + r} L_{2rn + r}  + 4(L_{2r}  + 1),\quad\mbox{by \eqref{equ.fpjkos2}}\,.
\end{split}
\]
\end{itemize}
\end{proof}
\begin{thm}\label{thm.qa1rrs8}
If $r$ and $n$ are integers such that $r$ is even, then
\[
F_{3r} \sum_{k = 1}^n {L_{2rk}^3 }  = F_{rn} L_{rn + r} (5L_{rn}F_{rn + r}F_{2rn + r}   + 4(L_{2r}  + 1))\,.
\]
\end{thm}
\begin{proof}
\[
\begin{split}
F_{3r} \sum_{k = 1}^n {L_{2rk}^3 }  &= F_{3r} \sum_{k = 1}^n {L_{6rk} }  + 3\frac{{F_{3r} }}{{F_r }}F_r \sum_{k = 1}^n {L_{2rk} }\\ 
 &= F_{3r} \sum_{k = 1}^n {L_{6rk} }  + 3(5F_r^2  + 3)F_r \sum_{k = 1}^n {L_{2rk} },\quad\mbox{by identity \eqref{equ.oyvdke0}}\\ 
 &= F_{3rn} L_{3rn + 3r}  + 3(5F_r^2  + 3)F_{rn} L_{rn + r},\quad\mbox{by Lemma \ref{thm.hnf079l}}\,. 
\end{split}
\]
Thus,
\[
\begin{split}
\frac{{F_{3r} \sum_{k = 1}^n {L_{2rk}^3 } }}{{F_{rn} L_{rn + r} }} &= \frac{{F_{3rn} L_{3rn + 3r} }}{{F_{rn} L_{rn + r} }} + 3(5F_r^2  + 3)\\
&= 5L_{rn} F_{rn + r} F_{2rn + r}  + L_{2r}  + 1 + 15F_r^2  + 9,\quad\mbox{by Lemma \ref{thm.c1bi001}}\\
&= 5L_{rn} F_{rn + r} F_{2rn + r}  + 4(L_{2r}  + 1),\quad\mbox{by \eqref{equ.fpjkos2} and \eqref{equ.gr4yegj}}\,.
\end{split}
\]
\end{proof}

\end{document}